\documentclass[11pt, reqno, psamsfonts]{amsart}
\pdfoutput=1

\usepackage{amssymb}
\usepackage{amsthm}
\usepackage{amsmath}
\usepackage{latexsym}
\usepackage[T1]{fontenc}
\usepackage[utf8]{inputenc}
\usepackage[russian, french, english]{babel}
\usepackage{graphicx}
\usepackage[justification=centering, labelfont=bf]{caption}
\usepackage{mathtools}
\usepackage{tikz}
\usepackage{amsbsy}
\usepackage[inline]{enumitem}
\usepackage{mathrsfs}
\usetikzlibrary{shapes,snakes}
\usetikzlibrary{arrows.meta}
\usetikzlibrary{decorations.pathmorphing}
\usetikzlibrary{patterns}
\usepackage{float}
\usepackage{array}
\usepackage{stmaryrd}
\usepackage{cancel} 
\usepackage{lmodern}
\usepackage{mathabx}
\usepackage{titlesec}
\usepackage[spacing=true,kerning=true,babel=true,tracking=true]{microtype}
\usepackage[shortcuts]{extdash}
\usepackage[foot]{amsaddr}
\usepackage[left=1in,right=1in,top=1in,bottom=1in,bindingoffset=0cm]{geometry}
\usepackage{centernot}
\usepackage[hidelinks]{hyperref}
\usepackage{xspace}
\usepackage{algpseudocode}
\usepackage{algorithm}
\usepackage[backend=biber, style=alphabetic, sorting=nyt, maxnames=100,backref=true]{biblatex}
\title{\sffamily Counting colorings of triangle-free graphs}
\date{}

\author{Anton~Bernshteyn}
\author{Tyler~Brazelton}
\author{Ruijia~Cao}
\author{Akum~Kang}

\address[A.B.~and T.B.]{\normalfont School of Mathematics, Georgia Institute of Technology, Atlanta, GA, USA}
\address[R.C.~and A.K.]{\normalfont School of Computer Science, Georgia Institute of Technology, Atlanta, GA, USA}

\email{bahtoh@gatech.edu}
\email{tbrazelton3@gatech.edu}
\email{rcao62@gatech.edu}
\email{kangakum@gatech.edu}

\thanks{Research of the first named author is partially supported by the NSF grant DMS-2045412.}

\newtheoremstyle{bfnote}%
{}{}%
{\slshape}{}%
{\bfseries}{\bfseries.}%
{ }%
{\thmname{#1}\thmnumber{ #2}\thmnote{ \ep{\normalfont{}#3}}}

\newtheoremstyle{claim}%
{}{}%
{\slshape}{}%
{\bfseries}{.}%
{ }%
{\thmname{#1}\thmnumber{ #2}\thmnote{ \ep{\normalfont{}#3}}}

\theoremstyle{bfnote}
\newtheorem{theo}{Theorem}[section]
\newtheorem*{theo*}{Theorem}
\newtheorem{prop}[theo]{Proposition}
\newtheorem{lemma}[theo]{Lemma}

\newtheorem*{claim*}{Claim}
\newtheorem{corl}[theo]{Corollary}
\newtheorem{conj}[theo]{Conjecture}

\newtheorem*{corl*}{Corollary}

\theoremstyle{definition}
\newtheorem{defn}[theo]{Definition}
\newtheorem*{defn*}{Definition}

\newtheorem*{exmp*}{Example}
\newtheorem{prob}[theo]{Problem}

\theoremstyle{remark}
\newtheorem*{ques*}{Question}
\newtheorem*{remk*}{Remark}

\theoremstyle{claim}
\newcounter{ForClaims}[section]

\newtheorem{smallclaim}{Claim}[ForClaims]

\newcommand*{\myproofname}{Proof}
\newenvironment{claimproof}[1][\myproofname]{\begin{proof}[#1]}{\end{proof}}

\newcommand{\0}{\varnothing}
\newcommand{\set}[1]{\{#1\}}
\newcommand{\N}{{\mathbb{N}}}

\renewcommand{\P}{\mathbb{P}}
\newcommand{\E}{\mathbb{E}}
\renewcommand{\epsilon}{\varepsilon}
\newcommand{\eps}{\epsilon}
\renewcommand{\phi}{\varphi}
\renewcommand{\theta}{\vartheta}
\renewcommand{\leq}{\leqslant}
\renewcommand{\geq}{\geqslant}
\newcommand{\defeq}{\coloneqq}
\newcommand{\im}{\mathrm{im}}
\newcommand{\bemph}[1]{{\normalfont#1}} 
\newcommand{\ep}[1]{\bemph{(}#1\bemph{)}} 

\newcommand{\emphd}[1]{{\fontseries{b}\selectfont\textsf{#1}}}
\newcommand{\pto}{\dashrightarrow}
\newcommand{\dom}{\mathrm{dom}}

\newcommand{\Bigparen}[1]{\left(#1\right)}
\newcommand{\biggparen}[1]{\biggl(#1\biggr)}
\newcommand{\ceiling}[1]{\lceil #1 \rceil}
\newcommand{\bigabs}[1]{|#1|}
\newcommand{\brac}[1]{\left[#1\right]}
\newcommand{\bigbrac}[1]{\Bigl[#1\Bigl]}

\newcommand{\col}[1]{C(#1)}
\newcommand{\pcol}[1]{C_\mathsf{p}(#1)}
\newcommand{\gcol}[1]{C_\mathsf{g}(#1)}
\newcommand{\dlist}[1]{L_f(#1)}
\newcommand{\bl}{\mathsf{blank}}
\newcommand{\Ext}{\mathsf{Ext}}
\numberwithin{equation}{section}

\titleformat{\section}[block]{\large\bfseries\sffamily}{\thesection.}{1ex}{}
\titleformat{\subsection}[block]{\bfseries\sffamily}{\thesubsection.}{1ex}{}
\titleformat{\subsubsection}[block]{\itshape}{\bfseries\sffamily\upshape\thesubsubsection.}{1ex}{}

\titlespacing*{\section}{0pt}{*2.5}{*1}
\titlespacing*{\subsection}{0pt}{*2.5}{*1}
\titlespacing*{\subsubsection}{0pt}{*2.5}{*1}

\makeatletter
\newcommand{\neutralize}[1]{\expandafter\let\csname c@#1\endcsname\count@}
\makeatother

\renewbibmacro{in:}{}

\renewbibmacro*{volume+number+eid}{%
	\printfield{volume}%
	\setunit*{\addnbspace}
	\printfield{number}%
	\setunit{\addcomma\space}%
	\printfield{eid}}

\DeclareFieldFormat[article]{volume}{\textbf{#1}\space}
\DeclareFieldFormat[article]{number}{\mkbibparens{#1}}

\DeclareFieldFormat{journaltitle}{#1,}
\DeclareFieldFormat[thesis]{title}{\mkbibemph{#1}\addperiod}
\DeclareFieldFormat[article, unpublished, thesis]{title}{\mkbibemph{#1},}
\DeclareFieldFormat[book]{title}{\mkbibemph{#1}\addperiod}
\DeclareFieldFormat[unpublished]{howpublished}{#1, }

\DeclareFieldFormat{pages}{#1}

\DeclareFieldFormat[article]{series}{Ser.~#1\addcomma}

\setlist{topsep=3pt,itemsep=3pt}

\usepackage{stackengine,amsmath}
\stackMath

\begin{document}
    \maketitle
    
    \begin{abstract}
        By a theorem of Johansson, every triangle-free graph $G$ of maximum degree $\Delta$ has chromatic number at most $(C+o(1))\Delta/\log \Delta$ for some universal constant $C > 0$. Using the entropy compression method, Molloy proved that one can in fact take $C = 1$. Here we show that for every $q \geq (1 + o(1))\Delta/\log \Delta$, the number $c(G,q)$ of proper $q$-colorings of $G$ satisfies
        \[
            c(G, q) \,\geq\, \left(1 - \frac{1}{q}\right)^m ((1-o(1))q)^n,
        \]
        where $n = |V(G)|$ and $m = |E(G)|$. Except for the $o(1)$ term, this lower bound is best possible as witnessed by random $\Delta$-regular graphs. When $q = (1 + o(1)) \Delta/\log \Delta$, our result yields the inequality
        \[
            c(G,q) \,\geq\, \exp\left((1 - o(1)) \frac{\log \Delta}{2} n\right),
        \]
        which improves an earlier bound of Iliopoulos and yields the optimal value for the constant factor in the exponent. Furthermore, this result implies the optimal lower bound on the number of independent sets in $G$ due to Davies, Jenssen, Perkins, and Roberts. An important ingredient in our proof is the counting method that was recently developed by Rosenfeld. As a byproduct, we obtain an alternative proof of Molloy's bound $\chi(G) \leq (1 + o(1))\Delta/\log \Delta$ using Rosenfeld's method in place of entropy compression (other proofs of Molloy's theorem using Rosenfeld's technique were given independently by Hurley and Pirot and Martinsson).
    \end{abstract}
    
    \section{Introduction}
    
    \subsection{Counting colorings}
    
    All graphs in this paper are finite, undirected, and simple. A celebrated theorem of Johansson \cite{Johansson} says that every triangle-free graph $G$ of maximum degree $\Delta$ satisfies $\chi(G) \leq (C + o(1))\Delta/\log \Delta$ for some universal constant $C > 0$. \ep{Here and throughout the paper, $o(1)$ indicates a function of $\Delta$ that approaches $0$ as 
    $\Delta \to \infty$.} The best currently known value for the constant $C$ is given by the following result of Molloy:
    
    \begin{theo}[{Molloy \cite{Molloy}}]\label{theo:Molloy}
        If $G$ is a triangle-free graph of maximum degree $\Delta$, then \[\chi(G) \,\leq\, (1 + o(1))\frac{\Delta}{\log\Delta}.\]
    \end{theo}
    
    In this paper we establish a lower bound on the number $c(G,q)$ of proper $q$-colorings of $G$ when $q \geq (1 + o(1))\Delta/\log\Delta$ \ep{i.e., when $G$ is $q$-colorable by Theorem~\ref{theo:Molloy}}. Here is our main result:
    
    \begin{theo}\label{theo:main}
        For each $\epsilon > 0$, there is $\Delta_0 \in \N$ such that the following holds. Let $G$ be a triangle-free graph of maximum degree at most $\Delta \geq \Delta_0$. Then, for every $q \geq (1 + \epsilon)\Delta/\log \Delta$, we have
        \begin{equation}\label{eq:main}
            c(G,q) \,\geq\, \left(1 - \frac{1}{q}\right)^m ((1-\delta)q)^n,
        \end{equation}
        where $n = |V(G)|$, $m = |E(G)|$, and $\delta = 4\exp(\Delta/q)/q$.
    \end{theo}
    
    It was shown by Csikv\'ari and Lin \cite[Corollary 1.2]{CL_Sidorenko} that if $G$ is bipartite, i.e., $G$ has no odd cycles, then $c(G,q) \geq (1-1/q)^mq^n$ for all $q \geq 1$ \ep{this is a special case of the so-called Sidorenko conjecture on the number of homomorphisms from a bipartite graph $G$ to a fixed graph $H$ \cite{Sidorenko}}. Our result asserts that approximately the same lower bound holds for triangle-free graphs $G$, under the assumption that $q \geq (1+o(1))\Delta/\log \Delta$. 
    
    The bound in Theorem~\ref{theo:main} has a natural probabilistic interpretation. Suppose $G$ is a graph with $n$ vertices and $m$ edges. If we assign a color from $[q] \defeq \set{1, \ldots, q}$ to each vertex of $G$ independently and uniformly at random, what is the probability $p(G, q)$ that the resulting $q$-coloring is proper? This problem is equivalent to computing $c(G,q)$ since $p(G,q) = c(G,q)/q^n$. For each edge $e \in E(G)$, let $D_e$ be the random event that the endpoints of $e$ get distinct colors. Then $\P[D_e] = 1 - 1/q$, 
    so if the events $(D_e \,:\,e \in E(G))$ were mutually independent, we would have
    \[p(G,q) \,=\, \left(1 - \frac{1}{q}\right)^m, \quad \text{or, equivalently,} \quad c(G,q) \,=\, \left(1 - \frac{1}{q}\right)^m q^n.\]
    Theorem~\ref{theo:main} says that when $G$ is triangle-free and $q \geq (1+\epsilon) \Delta/\log\Delta$, the actual value of $c(G,q)$ is not too much smaller than this ``naive'' bound. 
    Notice that, since $q \geq (1 + \epsilon)\Delta/\log \Delta$, 
    \[
        \delta \,=\, 4 \cdot \frac{\exp(\Delta/q)}{q} \,\leq\, \frac{4}{1+\epsilon} \cdot \log \Delta \cdot \Delta^{-\frac{\epsilon}{1 + \epsilon}} \,=\, o(1),
    \]
    which enables us to treat the factor $1 - \delta$ in \eqref{eq:main} as an error term. \ep{On the other hand, below the $\Delta/\log \Delta$ threshold, i.e., for $q < \Delta/\log \Delta$, the value $\delta$ tends to infinity as a function of $\Delta$.}
    
    It is natural to wonder how sharp our lower bound on $c(G,q)$ is. We show that it is optimal \ep{modulo the error term $1 - \delta$} for all values of $\Delta$ and $q$:
    
    \begin{theo}\label{theo:random}
        Fix positive integers $\Delta$ and $q$. For every sufficiently large $n \in \N$ such that $\Delta n$ is even, there exists a triangle-free $\Delta$-regular graph $G$ with
        \begin{equation}\label{eq:upper}
            c(G,q) \,\leq\, \left(1 - \frac{1}{q}\right)^{m} \, ((1 + \gamma)q)^n,
        \end{equation}
        where $n = |V(G)|$, $m = |E(G)| = \Delta n/2$, and $\gamma = 2\log n/n$.
    \end{theo}
    
    We prove Theorem~\ref{theo:random} in \S\ref{sec:upper} by showing that the bound \eqref{eq:upper} holds for the random $\Delta$-regular graph with high probability.
    
    Let us now explore some of the consequences that can be derived from Theorem~\ref{theo:main} by applying it to specific values of $q$. Perhaps the most natural regime to consider is when $q$ is close to $\Delta/\log \Delta$, i.e., when $q = (1+\epsilon)\Delta/\log\Delta$ for a small constant $\epsilon > 0$. Iliopoulos \cite[Theorem 1.2]{FI} showed that in this case $c(G,q)$ is exponentially large in $n$, i.e., $c(G,q) \geq e^{a n}$ for some constant $a > 0$ that only depends on $\epsilon$ and $\Delta$. Specifically, Iliopoulos's calculations yield the value of order $a = \Theta(\epsilon/\log \Delta)$. Using Theorem~\ref{theo:main}, we obtain the optimal value for the constant factor in the exponent, namely $a = (1+\epsilon/(1+\epsilon) - o(1))\log\Delta/2$, which significantly improves Iliopoulos's result (the optimality follows from Theorem~\ref{theo:random}): 
    
    \begin{corl}\label{corl:small_q}
        The following holds for each $\epsilon > 0$. 
        Let $G$ be an $n$-vertex triangle-free graph of maximum degree at most $\Delta$. 
        If $q \geq (1 + \epsilon) \Delta/\log \Delta$, then
        \[
            c(G,q) \,\geq\, \exp\left(\left(1 + \frac{\epsilon}{1+\epsilon} - o(1)\right) \frac{\log \Delta}{2} \, n\right).
        \]
    \end{corl}
    \begin{proof}
        A direct calculation using \eqref{eq:main} and the bounds $q \geq (1 + \epsilon) \Delta/\log \Delta$ and $m \leq \Delta n/2$. \ep{We are using that $\log(1 + x) \sim x$ for real $x \to 0$.}
    \end{proof}

    A curious consequence of Corollary~\ref{corl:small_q} is the optimal lower bound on the number of independent sets in triangle-free graphs due to Davies, Jenssen, Perkins, Roberts:
    
    \begin{corl}[{Davies--Jenssen--Perkins--Roberts \cite[Theorem 2]{indep}}]\label{corl:indep}
        Let $G$ be an $n$-vertex triangle-free graph of maximum degree at most $\Delta$. Then
        \[
            i(G) \,\geq\, \exp\left((1 - o(1)) \,\frac{\log^2 \Delta}{2\Delta}\, n\right),
        \]
        where $i(G)$ denotes the number of independent sets in $G$. 
    \end{corl}
    \begin{proof}
        Fix any $\epsilon > 0$ and set $q \defeq (1+\epsilon) \Delta/\log \Delta$. Since a proper $q$-coloring of $G$ is a sequence of $q$ independent sets in $G$ that partition $V(G)$, we have $c(G,q) \leq i(G)^q$. Therefore, by Corollary~\ref{corl:small_q},
        \[
            i(G) \,\geq\, c(G,q)^{1/q} \,\geq\, \exp\left(\left(1 -\frac{\epsilon^2}{(1+\epsilon)^2} - o(1)\right) \frac{\log^2 \Delta}{2\Delta} n\right).
        \]
        As $\epsilon$ can be taken arbitrarily small, the desired result follows.
    \end{proof}
    
    We find it intriguing that the crude way of bounding the number of colorings by counting independent sets employed in the above proof of Corollary~\ref{corl:indep} actually yields the optimal result (modulo the lower order term in the exponent).

    
    Theorem~\ref{theo:main} also has interesting consequences for larger values of $q$, e.g., for $q = \Delta + 1$:
    
    \begin{corl}\label{corl:large_q}
        Let $G$ be an $n$-vertex triangle-free graph of maximum degree at most $\Delta$. Then
        \[
            c(G,\Delta+1) \,\geq\, \left(\frac{\Delta}{\sqrt{e}} - O(1)\right)^{n}.
        \]
    \end{corl}
    \begin{proof}
        Follows by substituting $\Delta + 1$ for $q$ in \eqref{eq:main} and using the bound $m \leq \Delta n/2$.
    \end{proof}
    
    Even though every graph of maximum degree $\Delta$ is $(\Delta+1)$-colorable, the conclusion of Corollary~\ref{corl:large_q} may fail for graphs that are not triangle-free. For instance, if $G$ is a disjoint union of $n/(\Delta+1)$ cliques of size $\Delta+1$, then, using Stirling's formula, we obtain
    \[
        c(G, \Delta + 1) \,=\, \left((\Delta + 1)!\right)^{n/(\Delta + 1)} \,=\, \left(\frac{\Delta}{e} + o(\Delta)\right)^n,
    \]
    which is less than the bound in Corollary~\ref{corl:large_q} roughly by a factor of $\sqrt{e}^n$.
    

    \subsection{Counting DP-colorings}
    
    Molloy proved his Theorem~\ref{theo:Molloy} not just for the ordinary chromatic number $\chi(G)$, but also for the list-chromatic number $\chi_\ell(G)$. In fact, as shown in \cite{Bernshteyn}, the same upper bound holds in the more general setting of \emph{DP-coloring} \ep{also known as \emph{correspondence coloring}}, introduced by Dvo\v{r}\'ak and Postle \cite{DP}. Recall that in the context of list-coloring, each vertex $v$ of a graph $G$ is given its own list $L(v)$ of colors to choose from, and the goal is to find a \emphd{proper $L$-coloring} of $G$, i.e., a mapping $f$ such that $f(v) \in L(v)$ for all $v \in V(G)$ and $f(u) \neq f(v)$ whenever $uv \in E(G)$. \ep{Ordinary coloring is a special case of this when all lists are the same.} DP-coloring further generalizes list-coloring by allowing the identifications between the colors in the lists to vary from edge to edge. Formally, DP-coloring is defined using an auxiliary graph called a \emph{DP-cover}:
    
    \begin{defn}
    A \emphd{DP-cover} of a graph $G$ is a pair $\mathcal{H} = (L, H)$, where $H$ is a graph and $L$ is an assignment of subsets $L(v) \subseteq V(H)$ to the vertices $v \in V(G)$ satisfying the following conditions:
    \begin{itemize}
        \item The family of sets $(L(v) \,:\, v \in V(G))$ is a partition of $V(H)$.
        \item For each $v \in V(G)$, $L(v)$ is an independent set in $H$.
        \item For $u$, $v \in V(G)$, the edges of $H$ between $L(u)$ and $L(v)$ form a matching; this matching is empty whenever $uv \notin E(G).$
    \end{itemize}
    We call the vertices of $H$ \emphd{colors}. 
    For $\alpha \in V(H)$, we let $v_\alpha$ denote the \emphd{underlying vertex} of $\alpha$ in $G$, i.e., the unique vertex $v \in V(G)$ such that $\alpha \in L(v)$. If two colors $\alpha$, $\beta \in V(H)$ are adjacent in $H$, we say that they \emphd{correspond} to each other and write $\alpha \sim \beta$.
    
    An \emphd{$\mathcal{H}$-coloring} is a mapping $f \colon V(G) \to V(H)$ such that $f(v) \in L(v)$ for all $v \in V(G)$. Similarly, a \emphd{partial $\mathcal{H}$-coloring} is a partial map $f \colon V(G) \pto V(H)$ such that $f(v) \in L(v)$ for all $v \in \dom(f)$. 
    A \ep{partial} $\mathcal{H}$-coloring $f$ is \emphd{proper} if the image of $f$ is an independent set in $H$, i.e., if $f(u) \not \sim f(v)$ 
    for all $u$, $v \in \dom(f)$. 
    
    A DP-cover $\mathcal{H} = (L,H)$ is \emphd{$q$-fold} for some $q \in \N$ if $|L(v)| = q$ for all $v \in V(G)$. The \emphd{DP-chromatic number} of $G$, denoted by $\chi_{DP}(G)$, is the smallest $q$ such that $G$ admits a proper $\mathcal{H}$-coloring with respect to every $q$-fold DP-cover $\mathcal{H}$.
\end{defn}

    To see that list-coloring is a special case of DP-coloring, consider the following construction. Suppose that each vertex $v$ of a graph $G$ is given a list $L(v)$ of colors to choose from. Define \[L'(v) \,\defeq\, \set{(v, \alpha) \,:\, \alpha \in L(v)}\] \ep{thus, the sets $L'(v)$ for different vertices $v$ are disjoint} and let $H$ be the graph with vertex set \[V(H) \,\defeq\, \set{(v,\alpha) \,:\, v \in V(G), \, \alpha \in L(v)}\] in which vertices $(v,\alpha)$ and $(u,\beta)$ are adjacent if and only if $uv \in E(G)$ and $\alpha = \beta$. Then $\mathcal{H} \defeq (H, L')$ is a DP-cover of $G$ and there is a natural one-to-one correspondence between the proper $L$-colorings and the proper $\mathcal{H}$-colorings of $G$.
    
    We prove the following generalization of Theorem~\ref{theo:main}:
    
    \begin{theo}\label{theo:mainDP}
        For each $\epsilon > 0$, there is $\Delta_0 \in \N$ such that the following holds. Let $G$ be a triangle-free graph of maximum degree at most $\Delta \geq \Delta_0$. Then, for all $q \geq (1 + \epsilon)\Delta/\log \Delta$ and every $q$-fold DP-cover $\mathcal{H}$ of $G$, the number of proper $\mathcal{H}$-colorings of $G$ is at least 
        \[
            \left(1 - \frac{1}{q}\right)^m ((1-\delta)q)^n,
        \]
        where $n = |V(G)|$, $m = |E(G)|$, and $\delta = 4 \exp(\Delta/q)/q$.
    \end{theo}
    
    The problem of counting DP-colorings was studied by Kaul and Mudrock in \cite{Kaul}, where they introduced the \emphd{DP-color function} $P_{DP}(G, q)$. By definition, $P_{DP}(G,q)$ is the minimum number of proper $\mathcal{H}$-colorings of $G$ taken over all $q$-fold covers $\mathcal{H}$ of $G$. Using this terminology, we can say that Theorem~\ref{theo:mainDP} provides a lower bound on $P_{DP}(G,q)$ for triangle-free graphs $G$ of maximum degree $\Delta$ when $q \geq (1 + o(1))\Delta/\log \Delta$.
    
    An interesting feature of the lower bound given by Theorem~\ref{theo:mainDP} is that it is sharp \ep{modulo the error term $1 - \delta$} for \emph{every} graph $G$, as was shown by Kaul and Mudrock:
    
    \begin{theo}[{Kaul--Mudrock \cite[Proposition 16]{Kaul}}]\label{theo:DPsharp}
        For every graph $G$ with $n$ vertices and $m$ edges and every $q \geq 1$, there is a $q$-fold DP-cover $\mathcal{H}$ of $G$ such that the number of proper $\mathcal{H}$-colorings of $G$ is at most $(1-1/q)^m q^n$.
    \end{theo}
    

    
    \subsection{Overview of the proof}\label{subsec:ProofOverview}
    
    In this subsection we outline the key ideas that go into the proofs of our main results. For simplicity, we shall focus on Theorem~\ref{theo:main}; the more general argument needed to establish Theorem~\ref{theo:mainDP} in the DP-coloring setting is virtually the same, except for a few minor technical changes.
    
    Let $G$ be a triangle-free graph of maximum degree at most $\Delta$ and let $q \geq (1 + \epsilon) \Delta/\log \Delta$. Our approach is inspired by Molloy's proof of the bound $\chi(G) \leq q$ \ep{i.e., of Theorem~\ref{theo:Molloy}}. To explain Molloy's strategy, we need to introduce some notation and terminology. Let $f \colon V(G) \pto [q]$ be a proper partial $q$-coloring of $G$. For each vertex $v \in V(G)$, we let $L_f(v)$ be the set of all colors $\alpha \in [q]$ such that no neighbor of $v$ is colored $\alpha$. Also, for $\alpha \in [q]$, let $\deg_f(\alpha, v)$ be the number of uncolored neighbors $u$ of $v$ such that $\alpha \in L_f(u)$. Define the following numerical parameters:
    \[
        \ell \,\defeq\, \frac{q}{2\exp(\Delta/q)} \qquad \text{and} \qquad d \,\defeq\, \frac{q}{50\exp(\Delta/q)}. 
    \]
    The partial coloring $f$ is \emphd{good} if it satisfies the following two conditions:
    \begin{itemize}
        \item for every uncolored vertex $v$, $|L_f(v)| \geq \ell$; and
        \item for every uncolored vertex $v$ and $\alpha \in L_f(v)$, $\deg_f(\alpha, v) \leq d$.
    \end{itemize}
    In order to find a proper $q$-coloring of $G$, Molloy establishes two auxiliary results:
    \begin{enumerate}[label=\ep{\normalfont{}M\arabic*}]
        \item\label{item:Molloy_1} $G$ admits a good proper partial $q$-coloring.
        
        \item\label{item:Molloy_2} Every good partial coloring can be extended to a proper $q$-coloring of the entire graph $G$.
    \end{enumerate}
    Statement \ref{item:Molloy_2} is proved using the Lov\'asz Local Lemma and is by now standard \ep{its first appearance is in the paper \cite{Reed} by Reed; see also \cite[\S4.1]{MolloyReed} for a textbook treatment}. On the other hand, Molloy's proof of \ref{item:Molloy_1} was highly original and combined several novel ideas. In particular, it relied on a technique introduced by Moser and Tardos in \cite{MT} and called the \emph{entropy compression method} \ep{the name is due to Tao \cite{Tao}}. Initially designed as a means to establish an algorithmic version of the Lov\'asz Local Lemma, entropy compression has by now become an invaluable tool in the study of graph coloring; see, e.g., \cite{Esperet, BCGR, Duj} for a sample of its applications. An alternative approach---with the so-called Lopsided Lov\'asz Local Lemma taking the place of entropy compression---was developed by the first named author in \cite{Bernshteyn}. The ideas of \cite{Molloy} and \cite{Bernshteyn} have been pursued further by a number of researchers in order to strengthen and extend Theorem~\ref{theo:Molloy} in various ways \cite{fraction,locallistsize,DKPS}.
    
    Very recently, Rosenfeld \cite{Rosenfeld} discovered a remarkably simple new technique that can be used as a substitute for entropy compression. A number of applications of Rosenfeld's method to \ep{hyper}graph coloring appear in the paper \cite{WW} by Wanless and Wood, which also describes a general framework for applying Rosenfeld's technique to coloring problems. One benefit of Rosenfeld's approach \ep{in addition to its simplicity} is that it not only proves the existence of an object with certain properties \ep{such as a coloring}, but also gives a lower bound on the number of such objects. This makes it particularly well-suited for our purposes. As a byproduct of our proof of Theorem~\ref{theo:main}, we obtain a new simple proof of \ref{item:Molloy_1} \ep{and hence of Molloy's Theorem~\ref{theo:Molloy}} using Rosenfeld's technique in lieu of entropy compression or the Lopsided Lov\'asz Local Lemma. We should remark that entropy compression-style arguments can also be used to obtain counting results \ep{this is the approach taken by Iliopoulos in \cite{FI}}, and it is quite likely that bounds similar or even equivalent to ours can be established by other methods. However, we found that Rosenfeld's technique works especially well for this problem, and we believe that this application of inductive counting to graph coloring is interesting in its own right. We also note that a different proof of Molloy's theorem using Rosenfeld's method was given in \cite{Pirot} by Hurley and Pirot and simplified by Martinsson \cite{Martinsson} \ep{their work was carried out independently from ours}.
    
    Let us now describe the main steps in our argument in more detail. 
    
    \begin{enumerate}[wide]
        
        \item Besides the use of the entropy compression method, Molloy's proof of \ref{item:Molloy_1} involved another novel ingredient, namely a version of the coupon-collector theorem for elements drawn uniformly at random from sets of varying sizes \cite[Lemma 7]{Molloy}. Our proof uses this result as well. In fact, we need a slightly stronger version of it, because in our setting $q$ may be significantly larger than $(1+\epsilon)\Delta/\log \Delta$ and because we need the error bounds to be more precise. We state and prove this strengthening in \S\ref{sec:coupon}.
    
        \item Next, in \S\ref{sec:partial}, we give a lower bound on the number of proper partial colorings of $G$. This is done via an analysis of the greedy coloring algorithm. That is, we color the vertices of $G$ one by one, where each next vertex is either left uncolored or assigned an arbitrary color that has not yet been used by any of its neighbors. Using the coupon-collector result from \S\ref{sec:coupon}, we argue that, on average, each vertex will have many available colors to choose from, which yields the desired lower bound on the total number of proper partial colorings. The bound we obtain here is already sufficient to deduce the lower bound on the number of independent sets in $G$ given by Corollary~\ref{corl:indep}.
        
        \item In \S\ref{sec:Rosenfeld} we use a version of Rosenfeld's method to argue that a fairly large fraction of all proper partial colorings of $G$ are good \ep{in particular, a good coloring exists}. Combined with the result in \S\ref{sec:partial}, this yields a lower bound on the number of good colorings.
        
        \item As mentioned earlier, a simple application of the Lov\'asz Local Lemma shows that every good partial coloring $f$ can be extended to a proper coloring of $G$. We need to know not only that such an extension exists, but also how many such extensions there are. Thankfully, the Lov\'asz Local Lemma can be used to derive an explicit lower bound on the probability that a random extension of $f$ is proper, which can be translated into a lower bound on the number of such extensions. This is accomplished in \S\ref{sec:ext}.
        
        \item Finally, in \S\ref{sec:finish}, we combine all the above results to derive a lower bound on the number of proper colorings of $G$. Some care has to be taken because the same proper coloring of $G$ may arise as an extension of several good partial colorings. Nevertheless, we are able to use a double counting argument to account for this and obtain the desired result. Curiously, the double counting at this stage is the main contributor to the error term $1 - \delta$ in the statement of Theorems~\ref{theo:main} and \ref{theo:mainDP}.
    \end{enumerate}
    
    \section{Probabilistic preliminaries}{\label{sec: prob}}
    
    The following is a standard form of the Chernoff inequality:
    
    
    \begin{lemma}[{\cite[Theorem 2.3(b)]{McDiarmid1998}}]\label{lemma:Chernoff}
    Suppose that $X_1$, \ldots, $X_n$ are independent random variables with $0\le X_i \le 1$ for each $i$. Let $X \defeq \sum_{i=1}^nX_i$. Then, for any $s>0$,
    \[\P\left[X\ge (1+s)\E[X]\right] \,\le\, \exp\left(-\frac{s^2\E[X]}{2(1+s/3)}\right).\]
    \end{lemma}
    
    We also need a version of the Chernoff bound for negatively correlated random variables, introduced by Panconesi and Srinivasan \cite{panconesi}. 
    We say that $\set{0,1}$-valued random variables $X_1$, \ldots, $X_m$ are \emphd{negatively correlated} if for all $I \subseteq \set{1,2,\ldots,m}$, 
    \[
    \P \brac{\bigcap_{i \in I} \set{X_i = 1}} \,\leq\, \prod_{i\in I} \P\left[X_i = 1\right].
    \]
    
    \begin{lemma}[{\cite[Theorem 3.2]{panconesi}, \cite[Lemma 3]{Molloy}}]\label{lemma:NegativeChernoff}
    Let $X_1$, \ldots, $X_m$ be $\set{0,1}$-valued random variables. Set $Y_i \defeq 1-X_i$ and $X \defeq \sum_{i=1}^mX_i$. If $Y_1$, \ldots, $Y_m$ are negatively correlated, then
    \[
        \P[X<\E[X] - t] \,<\, \exp{\left(-\frac{t^2}{2\E{[X]}}\right)}.
    \]
    \end{lemma}
    
    We shall use the Lov\'asz Local Lemma in the following quantitative form:

    \begin{lemma}[{\cite[Lemma 5.1.1]{prob_method}}]\label{lemma:LLLgeneral}
    Let $ \mathcal {A}$ be a finite set of random events. For each $A \in {\mathcal {A}}$, let $\Gamma (A)$ be a subset of $\mathcal{A} \setminus \set{A}$ such that $A$ is mutually independent from the events in $\mathcal{A} \setminus (\Gamma(A) \cup \set{A})$. If there exists an assignment of reals $ x:{\mathcal {A}}\to [0,1)$ to the events such that
    \[ \forall A\in {\mathcal {A}} \ :\ \P[A] \,\leq\, x(A) \prod _{B\in \Gamma (A)}(1-x(B)),
    \]
    then the probability that no event in $\mathcal{A}$ happens is at least $\prod_{A \in \mathcal{A}}(1-x(A))$. 
    \end{lemma}
    
    More specifically, we will need the following consequence of Lemma~\ref{lemma:LLLgeneral}:
    
    \begin{corl}[Quantitative Symmetric Lov\'asz Local Lemma]\label{corollary:LLL}
        Let $ \mathcal {A}$ be a finite set of random events. For each $A \in {\mathcal {A}}$, let $\Gamma (A)$ be a subset of $\mathcal{A} \setminus \set{A}$ such that $A$ is mutually independent from the events in $\mathcal{A} \setminus (\Gamma(A) \cup \set{A})$. Suppose that for all $A \in \mathcal{A}$, $\P[A] \leq p$ and $|\Gamma(A)| \leq \mathcal{D}$, where $p \in [0,1)$ and $\mathcal{D} \in \N$. If $e p (\mathcal{D} + 1) \leq 1$, then
        \[
            \P\brac{\bigcap_{A \in \mathcal{A}} \overline{A}} \,\geq\, \left(1 - \frac{1}{\mathcal{D} + 1}\right)^{|\mathcal{A}|}.
        \]
    \end{corl}
    \begin{proof}
        Take $x(A) = 1/(\mathcal{D} + 1)$ in the statement of Lemma~\ref{lemma:LLLgeneral}.
    \end{proof}

    \section{Proof of Theorem~\ref{theo:mainDP}}\label{sec:proof}
    
    \subsection{Standing assumptions and notation}
    
    Throughout \S\ref{sec:proof}, we fix the following data:
    \begin{itemize}
        \item a real number $0 < \epsilon < 1$;
        \item an integer $\Delta$, assumed to be large enough as a function of $\epsilon$;
        \item an integer $q$ satisfying $q \geq (1+\epsilon)\Delta/\log\Delta$;
        \item a triangle-free graph $G$ of maximum degree at most $\Delta$ with $n$ vertices and $m$ edges;
        \item a $q$-fold DP-cover $\mathcal{H} = (L,H)$ of $G$.
    \end{itemize}
    As mentioned in \S\ref{subsec:ProofOverview}, we also define
    \[
        \ell \,\defeq\, \frac{q}{2\exp(\Delta/q)} \qquad \text{and} \qquad d \,\defeq\, \frac{q}{50\exp(\Delta/q)}. 
    \]
    
    The \emphd{neighborhood} $N(v)$ of a vertex $v \in V(G)$ is the set of all neighbors of $v$ in $G$. The \emphd{closed neighborhood} of $v$ is the set $N[v] \defeq N(v) \cup \set{v}$, and the \emphd{second neighborhood} $N^2[v]$ is the set of all vertices at distance at most $2$ from $v$. For a subset $U \subseteq V(G)$, we write $N_U(v) \defeq N(v) \cap U$ and $\deg_U(v) \defeq |N_U(v)|$. Given $\alpha \in V(H)$, the notation $N(\alpha)$, $N[\alpha]$, etc.~is defined analogously but with respect to the graph $H$ instead of $G$. For a set $U \subseteq V(G)$ and a vertex $x \in V(G) \setminus U$, we use $U + x$ to denote the set $U \cup \set{x}$.
    
    When $f$ is a partial function and $f(x)$ is {undefined} for some element $x$, we write $f(x) = \bl$. Given a partial $\mathcal{H}$-coloring $f$ of $G$ and $v \in V(G)$, we let
    \[
        L_f(v) \,\defeq\, \set{\alpha \in L(v) \,:\, N(\alpha) \cap \im(f) = \0}.
    \]
    Also, for each $\alpha \in V(H)$, we let
    \[
        \deg_f(\alpha) \,\defeq\, |\set{\beta \in N(\alpha) \,:\, f(v_\beta) = \bl \text{ and } \beta \in L_f(v_\beta)}|.
    \]
    \ep{Recall that $v_\beta \in V(G)$ here is the underlying vertex of the color $\beta$.}
    
    \subsection{A coupon-collector lemma}\label{sec:coupon}
    
    In this subsection, we establish a version of the coupon-collector theorem that slightly generalizes \cite[Lemma 7]{Molloy} by Molloy. Our argument closely follows Molloy's proof.
    
    \begin{lemma}[Coupon-collector]\label{lemma:coupon}
        Let $L_0$, $L_1$, \ldots, $L_k$ be finite sets, where $k \leq \Delta$ and $|L_0| = q$ \ep{there are no assumptions on $|L_i|$ for $i \in [k]$}. For each $i \in [k]$, let $M_i$ be a matching between $L_0$ and $L_i$. 
        For every $i \in [k]$, pick an element $f(i)$ uniformly at random from $L_i \cup \set{\bl}$, making the choices for different $i$ independently. This defines a random partial function on the set $[k]$. Let
        \[L_0' \,\defeq\, \set{\alpha \in L_0 \,:\, \text{$\alpha$ is not matched to any $f(i)$}},\]
        and, for each $\alpha \in L_0$, let
        \[
            \deg'(\alpha) \defeq |\set{i \in [k] \,:\, \text{$f(i) = \bl$ and $\alpha$ is matched to some $\beta \in L_i$}}|.
        \]
        Then the following statements are valid:
        \begin{enumerate}[label=\ep{\normalfont\alph*}]
            \item\label{item:Expectation} $\E[|L_0'|] \geq \dfrac{q}{\exp(k/q)} \geq 2\ell$.
        
            \item\label{item:ListConcentration} $\P\left[|L'_0|<\ell\right]< \exp\left(-\frac{1}{8} q^{\eps/2}\right)$.
            
            \smallskip
            
            \item\label{item:DegreeConcentration} $\P \left[\text{$\exists \alpha \in L_0'(v)$ such that $\deg'(\alpha)> d$}\right] < \exp\left(-\frac{1}{300}q^{\eps/2}\right)$.
        \end{enumerate}
    \end{lemma}
    \begin{proof}\stepcounter{ForClaims} \renewcommand{\theForClaims}{\ref{lemma:coupon}}
    Without loss of generality, we may assume that $L_i \neq \0$ for all $i \in [k]$. For each $\alpha \in L_0$, let $N_\alpha$ be the set of all indices $i \in [k]$ such that $\alpha$ is matched to some $\beta \in L_i$. 
    Define a quantity $\rho(\alpha)$ by
    \[
        \rho(\alpha) \,\defeq\, \sum_{i \in N_\alpha} \frac{1}{|L_i|}.
    \]
    Observe that, since each $M_i$ is a matching,
    \begin{equation}\label{eq:SumOfRho}
        \sum_{\alpha \in L_0} \rho(\alpha) \,\leq\, \sum_{i = 1}^k \sum_{\beta \in L_i} \frac{1}{|L_i|} \,=\, k.
    \end{equation}
    Notice also that, since $q \geq (1+\epsilon)\Delta/\log\Delta$, for large enough $\Delta$ we have
    \begin{equation}\label{eq:qLarge}
        \frac{q}{\exp(\Delta/q)} \,\geq\, q \,\Delta^{-\frac{1}{1+\epsilon}} \,\geq\, q^{\epsilon/2}.
    \end{equation}
    
    \smallskip
    
    \ref{item:Expectation} Using the inequality $1-1/(x+1) \geq \exp(-1/x)$ valid for all $x > 0$, we obtain 
    \begin{equation}
        \E\brac{|L'_0|} \,=\, \sum_{\alpha \in L_0} \prod_{i\in N_\alpha} \Bigparen{1-\frac{1}{|L_i| + 1}} \,
        \geq \,\sum_{\alpha \in L_0} \exp(-\rho(\alpha)). \label{eq:ExpOfL0}
    \end{equation}
    Applying Jensen's inequality to \eqref{eq:ExpOfL0} and using \eqref{eq:SumOfRho}, we get 
    \[
        \E\brac{|L'_0 |} \,\geq\, q \,\exp{\left(-\frac{k}{q}\right)}, 
    \]
    as desired. Note that, since $k \leq \Delta$, we also have $q/\exp(k/q) \geq q/\exp(\Delta/q) = 2\ell$.
    
    \smallskip
    
    \ref{item:ListConcentration}
    For $\alpha \in L_0$, let $X_{\alpha}$ be the indicator random variable of the event $\set{\alpha\in L'_0}$ and let $Y_\alpha \defeq 1 - X_\alpha$. 
    We claim that the random variables $(Y_{\alpha}\,:\, \alpha \in L_0)$ are negatively correlated:
    
    \begin{smallclaim}\label{claim:negative}
    For any $I\subseteq L_0$, $\P\brac{\bigcap_{\alpha \in I} \set{Y_\alpha = 1}} \leq \prod_{\alpha \in I} \P[Y_\alpha = 1]$.
    \end{smallclaim}
    \begin{claimproof}[Proof of Claim~\ref{claim:negative}]
    We first notice that for any $I\subseteq L_0$ and $\alpha' \in L_0 \setminus I$,
    \begin{equation}\label{eq:IndStep}
        \P\brac{\bigcap_{\alpha\in I} \set{Y_{\alpha} = 1} \,\middle\vert\, X_{\alpha'} = 1} \,\ge\, \P\brac{\bigcap_{\alpha\in I} \set{Y_{\alpha} = 1}}.
    \end{equation}
    To see this, for each $i \in [k]$ and $\alpha \in L_0$, let $L_{i, \alpha}$ be the set of all elements $\beta \in L_i$ such that $\alpha\beta \in M_i$ \ep{so $L_{i,\alpha}$ contains at most one element}. 
    To sample $f$ conditioned on the event $X_{\alpha'} = 1$, we pick each $f(i)$ uniformly at random from $(L_i \setminus L_{i, \alpha'}) \cup \set{\bl}$. As $L_{i,\alpha'} \cap L_{i,\alpha} = \0$ for all $\alpha \in I$ and $i \in [k]$, the removal of $L_{i,\alpha'}$ from $L_i$ does not decrease the probability that for each $\alpha \in I$, there is $i \in [k]$ with $f(i) \in L_{i,\alpha}$, so \eqref{eq:IndStep} holds.  
    Now, expanding the left hand side of \eqref{eq:IndStep}, we see that it is equivalent to 
    \begin{align*}
        1- \P\brac{X_{\alpha'} = 1 \, \middle\vert\, \bigcap_{\alpha \in I} \set{Y_{\alpha} = 1} } \leq 1- \P\brac{X_{\alpha'} = 1}
    \end{align*}
    Since $\set{X_{\alpha'} = 1}$ and $\set{Y_{\alpha'} = 1}$ are complementary events, we see that \eqref{eq:IndStep} is equivalent to
    \begin{equation}\label{eq:IndStep1}
        \P\brac{Y_{\alpha'} = 1 \,\middle\vert\, \bigcap_{\alpha\in I} \set{Y_{\alpha} = 1}} \,\le\, \P\brac{ Y_{\alpha'} = 1}
    \end{equation}
    Applying \eqref{eq:IndStep1} inductively establishes the claim.
    \end{claimproof}
    
    Recalling that $\E[|L_0'|] \geq 2\ell = q /\exp(\Delta/q)$, using Lemma~\ref{lemma:NegativeChernoff}, and invoking \eqref{eq:qLarge}, we obtain 
    \[
        \P\brac{|L_0'|<\ell} \,\leq\, \P\brac{|L_0'| < \frac{1}{2}\E[|L_0'|]} \,<\, \exp{\Bigparen{-\frac{\E\brac{|L_0'|}}{8}}}\,\le\, \exp\Bigparen{-\frac{q^{\epsilon/2}}{8}}.
    \]
    
    \smallskip
    
    \ref{item:DegreeConcentration} Consider any $\alpha \in L_0$. We will bound $\P[\alpha \in L_0']$ by considering two cases depending on whether $\rho(\alpha) \ge d/2$ or $\rho(\alpha) < d/2$. If $\rho(\alpha) \geq d/2$, then, using the inequality $1 - 1/(x+1) \leq \exp(-1/(2x))$ valid for all $x \geq 1$, we can write
    \begin{equation}\label{eq:first_bound}
        \P[\alpha \in L_0'] \,=\, \prod_{i \in N_\alpha} \left(1 - \frac{1}{|L_i| + 1}\right) \,\leq\, \exp\left(-\frac{\rho(\alpha)}{2}\right) \,\leq\, \exp\left(- \frac{d}{4}\right) \,\leq\, \exp\left(-\frac{q^{\epsilon/2}}{200}\right).
    \end{equation}
    If, on the other hand, $\rho(\alpha) < d/2$, then
    \[\E\brac{\deg'(\alpha)} \,=\, \sum_{i\in N_{\alpha}} \frac{1}{|L_i| + 1} \,\leq\, \rho(\alpha) \,<\, \frac{d}{2}.\]
    Since the values $f(i)$ for distinct $i$ are chosen independently, we may apply Lemma~\ref{lemma:Chernoff} to get
    \[
        \P\brac{\deg'(\alpha) \geq(1+s)\E[\deg'(\alpha)]} \,\leq\, \exp\left(-\frac{s^2\E[\deg'(\alpha)]}{2(1+s/3)}\right),
    \]
    for any $s > 0$. We may assume $\E[\deg'(\alpha)] > 0$ \ep{otherwise $\deg'(\alpha) = 0$ with probability $1$} and plug in the value $s = d/(2\E[\deg'(\alpha)])$, which yields
    \begin{equation}\label{eq:color_degree}
        \P{\bigbrac{\deg'(\alpha) > d}} \,\leq\, \P{\bigbrac{\deg'(\alpha) \geq \E\brac{\deg'(\alpha)} + d/2}} \,\le\, \exp\Bigparen{-\frac{3d}{16}} \,\leq\, \exp {\Bigparen{-\frac{3q^{\eps/2}}{800}}}.
    \end{equation}
    Since $3/800 < 1/200$, it follows from \eqref{eq:first_bound} and \eqref{eq:color_degree} that for all $\alpha \in L_0$,
   \begin{align*}
       \P\left[\alpha \in L_0' \text{ and } \deg'(\alpha) > d\right] \,&\leq\, \min \left\{\P\left[\alpha \in L_0'\right], \, \P{\bigbrac{\deg'(\alpha) > d}} \right\} \\
       &\leq \exp\left(-\frac{3q^{\epsilon/2}}{800}\right).
   \end{align*}
    Therefore, we may conclude that
    \begin{align*}
        \P \bigbrac{\exists \alpha\in \dlist{v} \text{ such that } \deg'(\alpha)> d} &\,\le\,   
        q \,\exp\left(-\frac{3q^{\epsilon/2}}{800}\right) \,<\, \exp\Bigparen{-\frac{q^{\eps/2}}{300}},
    \end{align*}
    assuming $\Delta$ is large enough.
    \end{proof}
    
    \subsection{Counting partial colorings}{\label{sec: partial}}\label{sec:partial}
    
    Let $\pcol{G}$ denote the set of all proper partial $\mathcal{H}$-colorings of $G$. Also, for a subset $U \subseteq V(G)$, let $\pcol{U}$ be the set of all proper partial $\mathcal{H}$-colorings $f \in \pcol{G}$ with $\dom(f) \subseteq U$. 
    In this subsection we establish a lower bound on $|\pcol{G}|$. We start with a lemma: 
    
    \begin{lemma}\label{lemma:pcolInduction}
        Suppose that $U\subseteq V(G)$ and $x\in V(G)\setminus U$.  Then
        \[
            |\pcol{U + x}| \,\geq\, q \, \exp\Bigparen{-\frac{\deg_U(x)}{q}} \, |\pcol{U}|.
        \]
    \end{lemma}
        \begin{proof}
            To begin with, observe that
        \begin{equation}\label{eq:addone}
            |\pcol{U + x}| \,=\, \sum_{f \in \pcol{U}} (|L_f(x)| + 1) \,\geq\, \sum_{f \in \pcol{U}} |L_f(x)|,
        \end{equation}
        since given a partial coloring $f \colon U \pto V(H)$, we can extend it to $U + x$ by assigning to $x$ an arbitrary color from $L_f(x) \cup \set{\bl}$. To get a lower bound on the right-hand side of \eqref{eq:addone}, we shall use Lemma~\ref{lemma:coupon}. For 
        a proper partial $\mathcal{H}$-coloring $g \colon U \setminus N_U(x) \pto V(H)$, let $\Ext_U(g)$ denote the set of all \emph{extensions} of $g$ to $U$, i.e., all proper partial $\mathcal{H}$-colorings $f \colon U \pto V(H)$ that agree with $g$ on $U \setminus N_U(x)$.
        Since $G$ is triangle-free, a coloring $f \in \Ext_U(g)$ is obtained by assigning to each $y \in N_U(x)$ an arbitrary color from $L_g(y) \cup \set{\bl}$. Therefore, we may apply Lemma~\ref{lemma:coupon}\ref{item:Expectation} with $k = \deg_U(x)$ and the sets $L(x)$ and $(L_g(y) \,:\, y \in N_U(x))$ playing the role of $L_0$, $L_1$, \ldots, $L_k$ to conclude that
        \[
            \frac{\sum_{f \in \Ext_U(g)} |L_f(x)|}{|\Ext_U(g)|} \,\geq\, q\,\exp\left(-\frac{\deg_U(x)}{q}\right).
        \]
        Now we can write
        %
        \begin{align*}
            \sum_{f \in \pcol{U}} |L_f(x)|
            \,&=\, \sum_{g \in \pcol{U\setminus N_U(x)}} \ \sum_{f\in \Ext_U(g)}\bigabs{L_{f}(x)}\\
            &=\, \sum_{g \in \pcol{U\setminus N_U(x)}} \ |\Ext_U(g)| \cdot \frac{\sum_{f \in \Ext_U(g)} |L_f(x)|}{|\Ext_U(g)|} \\
            &\geq\, q\,\exp\left(-\frac{\deg_U(x)}{q}\right) \ \sum_{g \in \pcol{U\setminus N_U(x)}} \ |\Ext_U(g)|\\
            &=\, q\,\exp\left(-\frac{\deg_U(x)}{q}\right) \, |\pcol{U}|.
        \end{align*}
        Combining this with \eqref{eq:addone} yields the desired result.
        \end{proof}
        
        \begin{corl}[Counting partial colorings]\label{corl:partial}
            We have
            \[|\pcol{G}| \,\geq\,  \left(1 - \frac{1}{q}\right)^m \, q^n.\]
            \end{corl}
        \begin{proof}
            Let $x_1$, \ldots, $x_n$ be an arbitrary ordering of the vertices of $G$. 
            Since, by definition, $|\pcol{\0}| = 1$, repeated applications of Lemma~\ref{lemma:pcolInduction} yield
            \begin{align*}
                \bigabs{\pcol{G}} \,&\geq\, q^n \ \prod_{k=1}^{n} \exp\Bigparen{-\frac{\deg_{\set{x_1,\ldots,x_{k-1}}}(x_k)}{q}} \\
                &=\, q^n \, \exp\left(- \frac{1}{q} \sum_{k = 1}^n \deg_{\set{x_1,\ldots,x_{k-1}}}(x_k) \right)\,=\,q^n \,\exp\Bigparen{-\frac{m}{q}} \,\geq\, \left(1 - \frac{1}{q}\right)^m \, q^n. \qedhere
            \end{align*}
        \end{proof}
        
        As mentioned in the introduction, Corollary~\ref{corl:partial} can already be used to derive the lower bound on the number of independent sets in $G$ given by Corollary~\ref{corl:indep}.
    
    \subsection{Counting good partial colorings}{\label{sec: good_partial}}\label{sec:Rosenfeld}
    
    Let $f \in \pcol{G}$ be a proper partial $\mathcal{H}$-coloring of $G$. We say that $f$ has a \emphd{flaw} at a vertex $v \in V(G)$ if $f(v) = \bl$ and at least one of the following holds:
    \begin{itemize}
        \item $|L_f(v)| < \ell$, or
        \item $\deg_f(\alpha) > d$ for some $\alpha \in L_f(v)$.
    \end{itemize}
    Let $\mathsf{Flaw}(f)$ be the set of all vertices $v \in V(G)$ such that $f$ has a flaw at $v$. If $\mathsf{Flaw}(f) = \0$, we say that $f$ is \emphd{good}. The set of all good partial $\mathcal{H}$-colorings of $G$ is denoted by $\gcol{G}$. Our goal in this subsection is to establish a lower bound on $|\gcol{G}|$. 
    
    Given a subset $U \subseteq V(G)$, we say that a partial $\mathcal{H}$-coloring $f \in \pcol{G}$ is \emphd{good on $U$} if $v \not \in \mathsf{Flaw}(f)$ for every vertex $v$ such that $N^2[v] \subseteq U$. Let $\gcol{U}$ denote the set of all $f \in \pcol{G}$ that are good on $U$. We emphasize that a coloring $f \in \gcol{U}$ is not required to belong to $\pcol{U}$, i.e., the domain of $f$ may not be a subset of $U$. However, whether or not $f$ is good on $U$ only depends on the restriction of $f$ to $U$ \ep{because whether or not $f$ has a flaw at $v$ is determined by the restriction of $f$ to $N^2[v]$}. Since every proper partial $\mathcal{H}$-coloring is vacuously good on the empty set, we have \[\gcol{\0} \,=\, \pcol{G}.\]
            
    \begin{lemma}\label{lemma:good}
        Suppose that $U\subseteq V(G)$ and $x\in V(G)\setminus U$.  Then
        \begin{align}\label{eq:indGood}
           |\gcol{U + x}| \,\geq\, \Bigparen{1-\exp\left(-\frac{q^{\eps/2}}{600}\right)} \, |\gcol{U}|.
        \end{align}
    \end{lemma}
    \begin{proof}\stepcounter{ForClaims} \renewcommand{\theForClaims}{\ref{lemma:good}}
        This is an inductive argument in the style of Rosenfeld \cite{Rosenfeld}. Note, however, that our application of Rosenfeld's method is somewhat different from the ones in \cite{Rosenfeld,WW}. Namely, we do not show that $|\gcol{U + x}|$ \emph{grows} by a certain factor compared to $|\gcol{U}|$, but rather that it \emph{does not shrink} too much. This difference appears crucial for our approach. We remark that in \cite{Pirot}, Hurley and Pirot prove Molloy's bound $\chi(G) \leq (1 + o(1))\Delta/\log \Delta$ using a more ``standard'' version of Rosenfeld's technique \ep{their argument does not refer to good partial colorings at all}. The Hurley--Pirot approach was further simplified by Martinsson in \cite{Martinsson}.
        
        We proceed by induction on $|U|$. So, fix $U\subseteq V(G)$ and suppose that \eqref{eq:indGood} holds when $U$ is replaced by any set of strictly smaller cardinality. Let $\mathcal{F}$ be the set of all $f \in \pcol{G}$ such that $f$ is good on $U$ but not on $U + x$. Then
        \[
            |\gcol{U + x}| \,=\, |\gcol{U}| - |\mathcal{F}|.
        \]
        For each $u \in V(G)$, define $\mathcal{F}_u \defeq \set{f \in \mathcal{F} \,:\, \text{$f$ has a flaw at $u$}}$. If $f \in \mathcal{F}$, then there must be a vertex $u \in \mathsf{Flaw}(f)$ such that $N^2[u] \subseteq U + x$. Since $N^2[u] \not \subseteq U$, this implies that $u \in N^2[x]$, and hence 
        \[
            |\mathcal{F}| \, \leq\, \sum_{u \in N^2[x]} |\mathcal{F}_u|.
        \]
        We will give an upper bound for $|\mathcal{F}_u|$ for each $u \in N^2[x]$. 
        
        \begin{smallclaim}\label{claim:Fu}
            Set $\eta \defeq \exp\left(-q^{\eps/2}/600\right)$ and $p \defeq \exp\left(-q^{\eps/2}/400\right)$. Then, for every $u \in N^2[x]$,
            \[
                |\mathcal{F}_u| \,\leq\, \frac{p |\gcol{U}|}{(1 - \eta)^\Delta}.
            \]
        \end{smallclaim}
        \begin{claimproof}[Proof of Claim~\ref{claim:Fu}]
            Let $\mathcal{S}$ be the set of all proper partial $\mathcal{H}$-colorings $g \colon V(G) \setminus N(u) \pto V(H)$ that are good on $U\setminus N(u)$ such that $g(u) = \bl$. For each $g \in \mathcal{S}$, let $\Ext(g)$ be the set of all \emph{extensions} of $g$, i.e., all proper partial $\mathcal{H}$-colorings of $G$ that agree with $g$ on $V(G) \setminus N(u)$. Also, let $\mathsf{FlawedExt}(g)$ be the set of all $f \in \Ext(g)$ that have a flaw at $u$. Since $G$ is triangle-free, a coloring $f \in \Ext(g)$ is obtained by assigning to each vertex $y \in N(u)$ an arbitrary color from $L_g(y) \cup \set{\bl}$. Thus, we may use parts \ref{item:ListConcentration} and \ref{item:DegreeConcentration} of Lemma~\ref{lemma:coupon} with $k = \deg(u)$ and the sets $L(u)$ and $(L_g(y) \,:\, y \in N(u))$ playing the role of $L_0$, $L_1$, \ldots, $L_k$ to conclude that 
            \[
                \frac{|\mathsf{FlawedExt}(g)|}{|\Ext(g)|} \,\leq\, \exp\left(-\frac{q^{\epsilon/2}}{8}\right) \,+\, \exp\left(-\frac{q^{\epsilon/2}}{300}\right) \,\leq\, p.
            \]
            Note that if $g \in \mathcal{S}$ and $f \in \Ext(g)$, then $f$ is good on $U \setminus N(u)$. Therefore,
            \begin{equation}\label{eq:p}
                |\mathcal{F}_u| \,\leq\, \sum_{g \in \mathcal{S}} |\mathsf{FlawedExt}(g)| \,\leq\, p \, \sum_{g \in \mathcal{S}} |\Ext(g)| \,\leq\, p\, |\gcol{U\setminus N(u)}|.
            \end{equation}
            Repeated applications of the inductive hypothesis show that
            \[
                |\gcol{U}| \,\geq\, (1-\eta)^\Delta \, |\gcol{U \setminus N(u)}|.
            \]
            Together with \eqref{eq:p}, this yields the desired bound on $|\mathcal{F}_u|$.
        \end{claimproof}
        
        Putting the above bounds together, we see that
        \begin{align*}
            |\gcol{U+x}| \,&\geq\, \bigabs{\gcol{U}} \,-\, \sum_{u\in N^2[x]} \bigabs{\mathcal{F}_u}\\
            &\ge\, \bigabs{\gcol{U}} \,-\, (\Delta^2+1) \cdot \frac{p}{(1 - \eta)^\Delta} \cdot |\gcol{U}|\\
            &=\, \Bigparen{1- \frac{p(\Delta^2+1)}{(1-\eta)^{\Delta}}} \, \bigabs{\gcol{U}} \\
            &\geq \left(1 - \eta\right) \bigabs{\gcol{U}},\phantom{\Bigparen{1- \frac{p(\Delta^2+1)}{(1-\eta)^{\Delta}}}}
        \end{align*}
        where the last inequality holds for $\Delta$ large enough.
        \end{proof}
        
        Keeping every vertex blank provides an example of a proper partial $\mathcal{H}$-coloring of $G$, so $\gcol{\0} = \pcol{G} \neq \0$. Therefore, applying Lemma~\ref{lemma:good} repeatedly gives
        \[
            |\gcol{G}| \,\geq\, \Bigparen{1-\exp\left(-\frac{q^{\eps/2}}{600}\right)}^n \,>\,0,
        \]
        which means that there must exist at least one good coloring of $G$. As mentioned in \S\ref{subsec:ProofOverview}, an application of the Lov\'asz Local Lemma shows that every good partial coloring can be extended to a proper coloring of the entire graph $G$, and thus $G$ is $\mathcal{H}$-colorable. Using Corollary~\ref{corl:partial} yields a better lower bound on $|\gcol{G}|$:
        
        \begin{corl}[Counting good partial colorings]\label{corl:good}
        We have
        \[
            \bigabs{\gcol{G}} \,\geq\,  \Bigparen{1-\exp\left(-\frac{q^{\eps/2}}{600}\right)}^n \, \Bigparen{1-\frac{1}{q}}^m \, q^n.
        \]
        \end{corl}
        \begin{proof}
            Use Corollary~\ref{corl:partial} and apply Lemma~\ref{lemma:good} $n$ times.
        \end{proof}
    
    \subsection{Completing a good coloring}{\label{sec:ext}}
    
    For each $g \in \gcol{G}$, let $\mathsf{Comp}(g)$ be the set of all proper $\mathcal{H}$-colorings $f \colon V(G) \to V(H)$ that \emphd{complete} $g$, meaning that $f(v) = g(v)$ whenever $g(v) \neq \bl$. 
    
    \begin{lemma}[Completing a good coloring]\label{lemma:comp}
        Let $g \in \gcol{G}$ be a coloring with $k$ blank vertices. Then
        \[
            \bigabs{\mathsf{Comp}(g)} \,\geq\, \Bigparen{\frac{\ell}{2}}^k.
        \]
    \end{lemma}
    \begin{proof}
        We will apply the Quantitative Local Lemma (Corollary~\ref{corollary:LLL}) to obtain a lower bound on $\bigabs{\mathsf{Comp}(g)}$. Set $\ell' \defeq \ceiling{\ell}$. By removing some colors from $L(v)$ for each $v \in V(G)$ if necessary, we may arrange that $|L_g(v)| = \ell'$ for every blank vertex $v$. Now we assign to each blank vertex $v$ a color from $L_g(v)$ uniformly at random. Let $f$ be the resulting coloring of $G$.
    
        Say that an edge $\alpha \beta \in E(H)$ is \emph{dangerous} if $g(v_\alpha) = g(v_\beta) = \bl$ and $\alpha \in L_g(v_\alpha)$, $\beta \in L_g(v_\beta)$, where $v_\alpha$, $v_\beta \in V(G)$ are the underlying vertices of $\alpha$ and $\beta$ respectively. For each dangerous edge $\alpha\beta\in E(H)$, let $A_{\alpha\beta}$ be the event that $f(v_\alpha) = \alpha$ and $f(v_\beta) = \beta$. By construction, $f$ is a proper $\mathcal{H}$-coloring of $G$ if and only if none of the events $A_{\alpha \beta}$ happen.
        
        Since $g$ is good, for every dangerous edge $\alpha \beta \in E(H)$, we have
        \[
            \P\brac{A_{\alpha\beta}} \,=\, \frac{1}{|L_g(V_\alpha)|\,|L_g(v_\beta)|} \,\le\, \frac{1}{\ell^2} \,\eqqcolon\, p.
        \]
        For every event $A_{\alpha\beta}$, let $\Gamma(A_{\alpha\beta})$ be the set of all events $A_{\gamma \delta}$ with $\set{v_\alpha, v_\beta} \cap \set{v_\gamma, v_\delta} \neq \0$. Then $A_{\alpha\beta}$ is mutually independent from the events not in $\Gamma(A_{\alpha\beta})$. Since $g$ is good, we have
        \[
            |\Gamma(A_{\alpha\beta})| \,\leq\, \sum_{\gamma \in L_g(v_\alpha)} \deg_g(\gamma) \,+\, \sum_{\delta\in L_g(v_\beta)} \deg_g{(\delta)} \,\le\, 2d\ell'.
        \]
        Therefore, we may apply Corollary \ref{corollary:LLL} with $\mathcal{D} \defeq 2d\ell'$. We now check that
        \[
            e p (\mathcal{D} + 1) \,=\, e \cdot \frac{1}{\ell^2} \cdot 2d\ell' \,=\, \frac{2e}{25} + o(1) \,<\, 1.
        \]
        Since there are at most $k d \ell'$ dangerous edges, Corollary~\ref{corollary:LLL} yields
        \[
            \P\left[\bigcap_{\alpha \beta} \overline{A_{\alpha\beta}}\right] \,\geq\, \left(1 - \frac{1}{2d\ell' + 1}\right)^{kd\ell'} \,\geq\, \exp\left(- \frac{k}{2}\right) \,\geq\, 2^{-k},
        \]
        where in the second inequality we use that $1 - 1/(x+1) \geq \exp(-1/x)$ for all $x > 0$. Finally, since $|L_g(v)| \geq \ell$ for every blank vertex $v$, we conclude that
        \[\bigabs{\mathsf{Comp}(g)} \, \geq \, 2^{-k} \cdot \ell^k \,=\, \Bigparen{\frac{\ell}{2}}^k. \qedhere\]
\end{proof}

    \subsection{Finishing the proof of Theorem~\ref{theo:mainDP}}{\label{sec:finish}}
    
    We are finally ready to complete the proof of Theorem~\ref{theo:mainDP}. Let $\col{G}$ denote the set of all proper $\mathcal{H}$-colorings of $G$. Set $\eta \defeq \exp\left(-q^{\eps/2}/600\right)$. By Corollary~\ref{corl:good}, we have
    \begin{equation}\label{eq:goodbound}
        |\gcol{G}| \,\geq\, (1 - \eta)^n \, \left(1 - \frac{1}{q}\right)^m \, q^n.
    \end{equation}
    Define a bipartite graph $B$ with parts $\gcol{G}$ and $\col{G}$ by joining each $g \in \gcol{G}$ to $f \in \col{G}$ if and only if $f \in \mathsf{Comp}(g)$, i.e., if $f(v) = g(v)$ for all $v$ such that $g(v) \neq \bl$. By Lemma~\ref{lemma:comp}, for every $g \in \gcol{G}$ with $k$ blank vertices,
    \[
        \deg_B(g) \,=\, |\mathsf{Comp}(g)| \,\geq\, \left(\frac{\ell}{2}\right)^k.
    \]
    On the other hand, if $f \in \col{G}$, then $f$ has at most ${n \choose k}$ neighbors in $\gcol{G}$ with $k$ blank vertices, since every neighbor of $f$ is obtained by uncoloring a subset of $V(G)$. Therefore,
    \begin{align}
        \gcol{G} \,=\, \sum_{g \in \gcol{G}} \sum_{f \in N_B(g)} \frac{1}{\deg_B(g)} \,&=\, \sum_{f \in \col{G}} \sum_{g \in N_B(f)} \frac{1}{\deg_B(g)} \nonumber\\
        &\leq\, \sum_{f \in \col{G}} \sum_{k=0}^n \left(\frac{2}{\ell}\right)^{k} {n \choose k} \,=\, \left(1 + \frac{2}{\ell}\right)^n |\col{G}|. \label{eq:colorBound}
    \end{align}
        Combining \eqref{eq:goodbound} and \eqref{eq:colorBound}, we see that, for large enough $\Delta$,
    \begin{align*}
    |\col{G}| \,&\geq\, \biggparen{\frac{1-\eta}{1+\frac{2}{\ell}}}^n \, \Bigparen{1-\frac{1}{q}}^m \, q^n \\
    &\geq\, \Bigparen{1-\frac{2}{\ell}}^n\,\Bigparen{1-\frac{1}{q}}^m\,q^n\,=\, (1-\delta)^n\,\Bigparen{1-\frac{1}{q}}^m\,q^n,
    \end{align*}
    where $\delta = 4\exp(\Delta/q)/q$, as desired.

    \section{Sharpness examples}\label{sec:upper}
    
    
    
    
    In this section we prove Theorem~\ref{theo:random}. 
    Before presenting the proof, we introduce some necessary definitions and notation, which are similar to those used in Wormald's survey paper \cite{wormald_1999}.
    Let $\mathcal{G}_{n, \Delta}$ be the uniform probability space of $\Delta$-regular graphs on $n$ vertices, where we assume that $\Delta n$ is even. The following procedure for sampling a graph $G \sim \mathcal{G}_{n,\Delta}$, known as the \emphd{pairing model}, was introduced by Bollob\'as \cite{Bollobas_regular}. Fix a set $W$ of $\Delta n$ points partitioned into $n$ cells $W_1$, \ldots, $W_n$, each of size $\Delta$. A perfect matching of the points in $W$ into $\Delta n/2$ pairs is called a \emphd{pairing}. Let $\mathcal{P}_{n,\Delta}$ be the uniform probability space of all pairings. To each $P\in \mathcal{P}_{n,\Delta}$, we associate a $\Delta$-regular multigraph $G(P)$ with vertex set $[n]$, where for each pair $xy \in P$ with $x \in W_i$ and $y \in W_j$, we add an edge between $i$ and $j$. Note that $G(P)$ may have loops and multiple edges. However, for fixed $\Delta$ and large enough $n$, the probability that $G(P)$ is simple is separated from $0$, meaning that 
    \[
        \P[\text{$G(P)$ is simple}] \,\geq\, c_\Delta,
    \]
    for all large enough $n$, where $c_\Delta > 0$ depends only on $\Delta$ \cite[Theorem 2.2]{wormald_1999}. It is not hard to see that, conditioned on the event that $G(P)$ is simple, the distribution of $G(P)$ coincides with $\mathcal{G}_{n,\Delta}$.
    
    
    Fix any $q$-coloring $f \colon [n] \to [q]$. Then $f$ defines a partition of the cells $W_1$, \ldots, $W_{\Delta}$ into $q$ color classes $C_1$, \ldots, $C_q$. The following algorithm generates a uniformly random pairing $P \sim \mathcal{P}_{n, \Delta}$:
    
    \begin{algorithm}[H]
    \caption{\textit{Generator}}\label{alg:cap}
    \begin{algorithmic}[1]
    \State $U \gets W$, $P \gets \0$;
    \For{$i = 1$, \ldots, $\Delta n/2$}
        \State \text{choose} $x$ \text{arbitrarily from} $C_{max}$, where $|C_{max}|$ is maximum among $|C_1|$, \ldots, $|C_q|$;
        \State \text{choose} $y$ \text{uniformly at random from} $U\setminus\{x\}$;
        \State $P \gets P \cup \{xy\}$, $U \gets U\setminus \{x,y\}$;
        \For{$j = 1$, \ldots, $q$}
            \State $C_j \gets C_j \setminus \set{x,y}$;
        \EndFor
    \EndFor
    \end{algorithmic}
    \end{algorithm}

    At the start of the $i$-th iteration of the outer loop in Algorithm~\ref{alg:cap}, we have $|U| = \Delta n - 2(i-1)$. By the choice of $C_{max}$, this implies that $|C_{max}| \geq (\Delta n - 2(i-1))/q$. Therefore,
    \[
        \P \brac{y \not\in C_{max}} \,=\, 1- \frac{|C_{max}| - 1}{|U| - 1} \,\leq \, \Bigparen{1-\frac{1}{q}} \Bigparen{1+\frac{1}{\Delta n-2i + 1}}.
    \]
    Since each step in the algorithm is independent from the previous ones, we have 
    \begin{align*}
        \P \brac{\text{$f$ is a proper coloring of $G(P)$}} \,&\le\, \prod_{i=1}^{\Delta n/2}  \Bigparen{1-\frac{1}{q}} \Bigparen{1+\frac{1}{\Delta n-2i + 1}}\\
        &\le\, \Bigparen{1-\frac{1}{q}}^{\Delta n/2} \exp{\left(\frac{1}{\Delta n-1}+\frac{1}{\Delta n-3}+\cdots+1\right)}. 
    \end{align*}
    Observe that, for $n$ large enough as a function of $\Delta$,
    \begin{align*}
        \frac{1}{\Delta n-1}+\frac{1}{\Delta n-3}+\cdots+1 \,&\leq\, \frac{1}{2}\left(\frac{1}{\Delta n - 1} + \frac{1}{\Delta n - 2} \right) + \frac{1}{2}\left(\frac{1}{\Delta n - 3} + \frac{1}{\Delta n - 4} \right) + \ldots + \frac{1}{2}(1 + 1) \\
        &= \frac{1}{2} \left(H_{\Delta n - 1} + 1 \right)\\ &\leq \frac{1}{2} \big(\log (\Delta n) + 2\big)\\
        &< \log n,
    \end{align*}
    where $H_{\Delta n - 1} \defeq 1/(\Delta n - 1) + 1/(\Delta n - 2) + \cdots + 1$ is the $(\Delta n - 1)$-th harmonic number. Thus, we may conclude that
    \[
        \P \brac{\text{$f$ is a proper coloring of $G(P)$}} \,\leq\, \Bigparen{1-\frac{1}{q}}^{\Delta n/2} n.
    \]
    
    Now let $X$ be the random variable equal to the number of proper $q$-colorings of $G(P)$. Then 
    \[\E\brac{X} \,=\, \sum_{f} \P \brac{\text{$f$ is a proper coloring of $G(P)$}} \,\le\, \Bigparen{1-\frac{1}{q}}^{\Delta n /2} n q^n.\]
    Set $\gamma \defeq 2(\log n)/n$. By Markov's inequality, we have 
    \begin{equation}\label{eq:Markov}
        \P \brac{X \geq (1+\gamma)^n \Bigparen{1-\frac{1}{q}}^{\Delta n/2} q^n } \,\le\, \frac{n }{(1+\gamma)^n}.
    \end{equation}
    The right-hand side of \eqref{eq:Markov} approaches $0$ as $n\to \infty$.
    Therefore, we see that
    \[
        X \,\leq\, (1+\gamma)^n \Bigparen{1-\frac{1}{q}}^{\Delta n /2} q^n
    \]
    asymptotically almost surely. On the other hand, as $n \to \infty$, the probability that $G(P)$ is simple and triangle-free approaches a positive constant depending only on $\Delta$ \cite[Theorem 2.12]{wormald_1999}. Thus, the number of proper $q$-colorings of a random regular graph $G \sim \mathcal{G}_{n ,\Delta}$ is at most
    \[
        (1+\gamma)^n \Bigparen{1-\frac{1}{q}}^{\Delta n /2} q^n
    \]
    asymptotically almost surely, and Theorem~\ref{theo:random} follows.
    
    \section{Open problems}
    
    It is not known if the constant factor in Molloy's Theorem~\ref{theo:Molloy} is optimal. Ignoring the lower order terms, the best known lower bound on the chromatic number for triangle-free graphs $G$ of maximum degree $\Delta$ is $(1/2 + o(1))\Delta/\log \Delta$ due to Frieze and \L{}uczak \cite{FL}, which holds for random $\Delta$-regular graphs with high probability. It is therefore possible that the conclusion of Theorem~\ref{theo:main} remains valid for $q \geq (1/2 + o(1))\Delta/\log\Delta$. To challenge the reader, we state this as a conjecture:
    
    \begin{conj}\label{conj:1/2}
        If $G$ is a triangle-free graph of maximum degree $\Delta$ and $q \geq (1/2+\epsilon)\Delta/\log\Delta$ for some $\epsilon > 0$, then
        \[
            c(G,q) \,\geq\, \left(1 - \frac{1}{q}\right)^m ((1-o_\epsilon(1))q)^n,
        \]
        where $n = |V(G)|$, $m = |E(G)|$, and $o_\epsilon(1)$ stands for a function of $\Delta$ and $\epsilon$ that approaches $0$ as $\Delta$ tends to $\infty$ while $\epsilon$ remains fixed.
    \end{conj}
    
    A proof of Conjecture~\ref{conj:1/2} would be an incredibly ambitious result, since under its assumptions, proving the bound $c(G,q) > 0$ \ep{i.e., $\chi(G) \leq q$}, or even $\alpha(G) \geq n/q$, is already considered a very hard open problem. This makes Conjecture~\ref{conj:1/2} a good target for a \emph{dis}proof, which may be more feasible than obtaining a new lower bound on $\chi(G)$ or a new upper bound on $\alpha(G)$.
    
    There is some evidence for Conjecture~\ref{conj:1/2} coming from random graphs. Bapst, Coja-Oghlan, Hetterich, Rassmann, and Vilenchik \cite[Theorem 1.1]{Cond1} showed that the conclusion of Conjecture~\ref{conj:1/2} holds with high probability for the Erd\H{o}s--R\'enyi random graph $\mathcal{G}(n, \Delta/n)$ when $\Delta$ is large enough. This result was later extended to all $\Delta \geq 3$ by Coja-Oghlan, Krzakala, Perkins, and Zdeborová \cite[Theorem 1.2]{Cond2}. We are not aware of an analogous result for the random $\Delta$-regular graph $\mathcal{G}_{n, \Delta}$, but it seems plausible that it could be derived using the methods of \cite{Cond3}. 
    
    The value $\delta = 4\exp(\Delta/q)/q$ of the error term in Theorem~\ref{theo:main} ``blows up'' when $q < \Delta/\log\Delta$. This means that proving Conjecture~\ref{conj:1/2} would likely require reducing the value of the error term even for $q \geq (1+o(1))\Delta/\log\Delta$. We feel that this is an interesting problem in its own right:
    
    \begin{prob}
        Can the error term $\delta$ in the statement of Theorem~\ref{theo:main} be asymptotically improved?
    \end{prob}
    
    \subsubsection*{Acknowledgments}
    
    We thank Hemanshu Kaul and Will Perkins for helpful comments on an earlier version of this paper and for drawing our attention to several relevant references. We are also grateful to the anonymous referees for their feedback.
    
    \printbibliography

@BOOK{MolloyReed,
    AUTHOR = "Molloy, M. and Reed, B.",
    TITLE = "{Graph Colourings and the Probabilistic Method}",
    PUBLISHER = "Springer",
    YEAR = "2002",
}

@article{Reed,
	author = {B. Reed},
	title = {The list colouring constants},
	journaltitle = {J. Graph Theory},
	volume = {31},
	date = {1999},
	pages = {149--153},
}

@article{MT,
	author = {R. Moser and G. Tardos},
	title = {A constructive proof of the general Lov\'{a}sz Local Lemma},
	journaltitle = {J. ACM},
	date = {2010},
	volume = {57},
	number = {2},
}

@unpublished{Tao,
	author = {T. Tao},
	title = {Moser's entropy compression argument},
	date = {2009},
	howpublished = {What's new, \url{https://terrytao.wordpress.com/2009/08/05/mosers-entropy-compression-argument/} (blog post)},
}

@article{Esperet,
	author = {L. Esperet and A. Parreau},
	title = {Acyclic edge-coloring using entropy compression},
	journaltitle = {European J. Combin.},
	volume = {34},
	number = {6},
	date = {2013},
	pages = {1019--1027},
}

@article{BCGR,
	author = {B. Bosek and S. Czerwi{\'{n}}ski and J. Grytczuk and P. Rz{\k{a}}{\.{z}}ewski},
	title = {Harmonious coloring of uniform	hypergraphs},
	journaltitle = {Applicable Analysis and Disc. Math.},
	date = {2016},
	volume = {10},
	pages = {73--87},
}

@article{Duj,
	author = {V. Dujmovi{\'{c}} and G. Joret and J. Kozik and D.R. Wood},
	title = {Nonrepetitive colouring via entropy compression},
	journaltitle = {Combinatorica},
	date = {2016},
	volume = {36},
	pages = {661--686},
}

@ARTICLE{Bernshteyn,
	AUTHOR = "Bernshteyn, A.",
	TITLE = "{The Johansson-Molloy theorem for DP-coloring}",
	JOURNAL = "Rand. Struct. Algor.",
	YEAR = "2019",
	volume = {54},
	pages = {653--664},
}

@article{locallistsize,
    author = {E. Davies and R. de Joannis de Verclos and R.J. Kang and F. Pirot},
    title = {Coloring triangle-free graphs with local list sizes},
    journaltitle = {Rand. Struct. Algor.},
    volume = {57},
    date = {2020},
    number = {3},
    pages = {730--744},
}

@article{fraction,
	author = {M. Bonamy and T. Kelly and P. Nelson and L. Postle},
	title = {Bounding $\chi$ by a fraction of $\Delta$ for graphs without large cliques},
	date = {2022},
	journaltitle = {J. Combin. Theory},
    series = {B},
	volume = {157},
	pages = {263--282},
}

@unpublished{DKPS,
    author = {E. Davies and R.J. Kang and F. Pirot and J.-S. Sereni},
    title = {Graph structure via local occupancy},
	howpublished = {\url{https://arxiv.org/abs/2003.14361} (preprint)},
	date = {2020},
}

@report{Johansson,
	author = {A. Johansson},
	title = {Asymptotic choice number for triangle free graphs},
	type = {Technical Report 91--95},
	institution = {DIMACS},
	date = {1996},
}

@article{Rosenfeld,
    author = {M. Rosenfeld},
    title = {Another approach to non-repetitive colorings of graphs of bounded degree},
    journaltitle = {Electron. J. Combin.},
    volume = {27},
    date = {2020},
    number = {3},
    pages = {Paper No. 3.43, 16 pp.},
}

@article{Molloy,
	author = {M. Molloy},
	title = {The list chromatic number of graphs with small clique number},
	journaltitle = {J. Combin. Theory},
	series = {B},
	volume = {134},
	pages = {264--284},
	date = {2019},
}

@article{indep,
	author = {E. Davies and M. Jenssen and W. Perkins and B. Roberts},
	title = {On the average size of independent sets in triangle-free graphs},
	journaltitle = {Proc. Amer. Math. Soc.},
	volume = {146},
	pages = {111--124},
	date = {2018},
}

@unpublished{WW,
    author = {I. Wanless and D. Wood},
    title = {A general framework for hypergraph colouring},
	howpublished = {\url{https://arxiv.org/abs/2008.00775} (preprint)},
	date = {2020},
}

@ARTICLE{DP,
	AUTHOR = "Dv\v{o}r\'ak, Z. and Postle, L.",
	TITLE = "{Correspondence coloring and its application to list-coloring planar graphs without cycles of lengths 4 to 8}",
	JOURNAL = "J. Combin. Theory",
	series = {B},
% 	archivePrefix = "arXiv",
% 	note = {arXiv:1508.03437},
	date = "2018",
	volume = {129},
	pages = {38--54},
	%MONTH = "mar",
}

@Inbook{McDiarmid1998,
author={McDiarmid, C.},
editor={Habib, M.
and McDiarmid, C.
and Ramirez-Alfonsin, J.
and Reed, B.},
title={Concentration},
bookTitle={Probabilistic Methods for Algorithmic Discrete Mathematics},
year={1998},
publisher={Springer Berlin Heidelberg},
address={Berlin, Heidelberg},
pages={195-248},
isbn={978-3-662-12788-9},
doi={10.1007/978-3-662-12788-9_6},
url={https://doi.org/10.1007/978-3-662-12788-9_6}
}

@book{prob_method,
author = {Alon, N. and Spencer, J.H.},
title = {The Probabilistic Method},
year = {2016},
isbn = {1119061954},
publisher = {Wiley Publishing},
edition = {4th}
}

@inbook
{wormald_1999, place={Cambridge}, series={London Mathematical Society Lecture Note Series}, title={Models of Random Regular Graphs}, DOI={10.1017/CBO9780511721335.010}, booktitle={Surveys in Combinatorics, 1999}, publisher={Cambridge University Press}, author={Wormald, N. C.}, editor={Lamb, J.D. and Preece, D.A.}, year={1999}, pages={239–298}, collection={London Mathematical Society Lecture Note Series}}

@article{panconesi,
author = {Panconesi, A. and Srinivasan, A.},
title = {Randomized distributed edge coloring via an extension of the Chernoff--Hoeffding bounds},
journal = {SIAM Journal on Computing},
volume = {26},
number = {2},
pages = {350-368},
year = {1997},
doi = {10.1137/S0097539793250767},

URL = { 
        https://doi.org/10.1137/S0097539793250767
    
},
eprint = { 
        https://doi.org/10.1137/S0097539793250767
    
}

}

@article{Bollobas_regular,
	author = {B. Bollob{\'{a}}s},
	title = {A probabilistic proof of an asymptotic formula for the number of labelled regular graphs},
	journaltitle = {Eur. J. Combin.},
	volume = {1},
	pages = {311--316},
	date = {1980},
}

@article{CL_Sidorenko,
	author = {P. Csikv{\'{a}}ri and Z. Lin},
	title = {Sidorenko's conjecture, colorings and independent sets},
	journaltitle = {Electron. J. Comb.},
	volume = {24},
	number = {1},
	pages = {P1.2},
	date = {2017},  
}

@article{Sidorenko,
	author = {A. Sidorenko},
	title = {A correlation inequality for bipartite graphs},
	journaltitle = {Graphs Combin.},
	volume = {9},
	pages = {201--204},
	date = {1993},  
}

@article{Kaul,
	author = {H. Kaul and J.A. Mudrock},
	title = {On the chromatic polynomial and counting DP-colorings of graphs},
	journaltitle = {Adv. Appl. Math.},
	volume = {123},
	number = {1},
	pages = {102131},
	date = {2021},  
}

@unpublished{Pirot,
	author = {E. Hurley and F. Pirot},
	title = {A first moment proof of the Johansson--Molloy theorem},
	date = {2021},
	howpublished = {\url{https://arxiv.org/abs/2109.15215} (preprint)},
}

@article{FL,
	author = {A. Frieze and T. {\L}uczak},
	title = {On the independence and chromatic numbers of random regular graphs},
	journaltitle = {J. Combin. Theory},
	series = {B},
	volume = {54},
	number = {1},
	pages = {123--132},
	date = {1992},
}

@article{Cond1,
	author = {V. Bapst and A. Coja-Oghlan and S. Hetterich and F. Rassmann and D. Vilenchik},
	title = {The condensation phase transition in random graph coloring},
	journaltitle = {Commun. Math. Phys.},
	volume = {341},
	pages = {543--606},
	date = {2016},
}

@article{Cond2,
	author = {A. Coja-Oghlan and F. Krzakala and W. Perkins and L. Zdeborová},
	title = {Information-theoretic thresholds from the cavity method},
	journaltitle = {Adv. Math.},
	volume = {333},
	pages = {694--795},
	date = {2018},
}

@article{Cond3,
	author = {A. Coja-Oghlan and C. Efthymiou and S. Hetterich},
	title = {On the chromatic number of random regular graphs},
	journaltitle = {J. Combin. Theory},
	series = {B},
	volume = {116},
	pages = {367--439},
	date = {2016},
}

@article{FI,
	author = {F. Iliopoulos},
	title = {Commutative algorithms approximate the LLL-distribution},
	journaltitle = {Approximation, Randomization, and Combinatorial Optimization (APPROX/RANDOM)},
	date = {2018},
	pages = {44:1--44:20},
	addendum = {Full version: \url{https://arxiv.org/abs/1704.02796}},
}

@unpublished{Martinsson,
	author = {A. Martinsson},
	title = {A simplified proof of the Johansson--Molloy Theorem using the Rosenfeld counting method},
	date = {2021},
	howpublished = {\url{https://arxiv.org/abs/2111.06214} (preprint)},
}
\end{document}